\documentclass{article}

\usepackage{arxiv}

\usepackage[utf8]{inputenc} 
\usepackage[T1]{fontenc}    
\usepackage{hyperref}       
\usepackage{url}            
\usepackage{booktabs}       
\usepackage{amsfonts}       
\usepackage{nicefrac}       
\usepackage{microtype}      
\usepackage{cleveref}       
\usepackage{graphicx}
\usepackage{natbib}
\usepackage{doi}

\usepackage{amsthm,amssymb,float}
\usepackage{color}
\theoremstyle{definition}
\newtheorem{theorem}{Theorem}[subsection]
\newtheorem{definition}[theorem]{Definition}
\newtheorem{example}[theorem]{Example}
\newtheorem{lemma}[theorem]{Lemma}
\newtheorem{prop}[theorem]{Proposition}
\newtheorem{corollary}[theorem]{Corollary}
\newtheorem{question}[theorem]{Question}

\newcommand{\R}{\mathbb{R}}
\newcommand{\N}{\mathbb{N}}
\newcommand{\Z}{\mathbb{Z}}
\newcommand{\SR}{\textnormal{SR}}

\usepackage{tikz,tkz-euclide,calc,xfp,lipsum,xstring,mathdots}
\newcommand{\OrigTop}[4]{
    \tkzDefShiftPoint[#1](120:#2){\fpeval{#1 * 2 + 3}}
    \tkzDefShiftPoint[#1](60:#2){\fpeval{#1 * 2 + 4}}
    \IfEqCase{#3}{
    {fill}{
    \tkzDrawPolygon[fill = gray!20](#1,\fpeval{#1 * 2 + 3},\fpeval{#1 * 2 + 4})
    }
    {nofill}{\tkzDrawPolygon(#1,\fpeval{#1 * 2 + 3},\fpeval{#1 * 2 + 4})}
    }
    \tkzDrawPoints(\fpeval{#1 * 2 + 3},\fpeval{#1 * 2 + 4})
    \IfEqCase{#4}{
    {end}{
    \tkzLabelPoint[above left](\fpeval{#1 * 2 + 3}){$\ddots$}
    \tkzLabelPoint[above right](\fpeval{#1 * 2 + 4}){$\iddots$}
    }
    {mid}{}
    }
}
\newcommand{\OrigRight}[4]{
    \tkzDefShiftPoint[#1](0:#2){\fpeval{#1 * 2 + 3}}
    \tkzDefShiftPoint[#1](-60:#2){\fpeval{#1 * 2 + 4}}
    \IfEqCase{#3}{
    {fill}{
    \tkzDrawPolygon[fill = gray!20](#1,\fpeval{#1 * 2 + 3},\fpeval{#1 * 2 + 4})
    }
    {nofill}{\tkzDrawPolygon(#1,\fpeval{#1 * 2 + 3},\fpeval{#1 * 2 + 4})}
    }
    \tkzDrawPoints(\fpeval{#1 * 2 + 3},\fpeval{#1 * 2 + 4})
    \IfEqCase{#4}{
    {end}{
    \tkzLabelPoint[above right](\fpeval{#1 * 2 + 3}){$\iddots$}
    \tkzLabelPoint[below](\fpeval{#1 * 2 + 4}){$\vdots$}
    }
    {mid}{}
    }
}

\newcommand{\OrigLeft}[4]{
    \tkzDefShiftPoint[#1](180:#2){\fpeval{#1 * 2 + 3}}
    \tkzDefShiftPoint[#1](240:#2){\fpeval{#1 * 2 + 4}}
    \IfEqCase{#3}{
    {fill}{
    \tkzDrawPolygon[fill = gray!20](#1,\fpeval{#1 * 2 + 3},\fpeval{#1 * 2 + 4})
    }
    {nofill}{\tkzDrawPolygon(#1,\fpeval{#1 * 2 + 3},\fpeval{#1 * 2 + 4})}
    }
    \tkzDrawPoints(\fpeval{#1 * 2 + 3},\fpeval{#1 * 2 + 4})
    \IfEqCase{#4}{
    {end}{
    \tkzLabelPoint[above left](\fpeval{#1 * 2 + 3}){$\ddots$}
    \tkzLabelPoint[below](\fpeval{#1 * 2 + 4}){$\vdots$}
    }
    {mid}{}
    }
}

\newcommand{\FlipBottom}[3]{
    \tkzDefShiftPoint[#1](-60:#2){\fpeval{#1 * 2 + 3}}
    \tkzDefShiftPoint[#1](-120:#2){\fpeval{#1 * 2 + 4}}
    \IfEqCase{#3}{
    {fill}{
    \tkzDrawPolygon[fill = gray!20](#1,\fpeval{#1 * 2 + 3},\fpeval{#1 * 2 + 4})
    }
    {nofill}{\tkzDrawPolygon(#1,\fpeval{#1 * 2 + 3},\fpeval{#1 * 2 + 4})}
    }
    \tkzDrawPoints(\fpeval{#1 * 2 + 3},\fpeval{#1 * 2 + 4})
}

\newcommand{\FlipLeft}[3]{
    \tkzDefShiftPoint[#1](-180:#2){\fpeval{#1 * 2 + 3}}
    \tkzDefShiftPoint[#1](120:#2){\fpeval{#1 * 2 + 4}}
    \IfEqCase{#3}{
    {fill}{
    \tkzDrawPolygon[fill = gray!20](#1,\fpeval{#1 * 2 + 3},\fpeval{#1 * 2 + 4})
    }
    {nofill}{\tkzDrawPolygon(#1,\fpeval{#1 * 2 + 3},\fpeval{#1 * 2 + 4})}
    }
    \tkzDrawPoints(\fpeval{#1 * 2 + 3},\fpeval{#1 * 2 + 4})
}

\newcommand{\FlipRight}[3]{
    \tkzDefShiftPoint[#1](0:#2){\fpeval{#1 * 2 + 3}}
    \tkzDefShiftPoint[#1](60:#2){\fpeval{#1 * 2 + 4}}
    \IfEqCase{#3}{
    {fill}{
    \tkzDrawPolygon[fill = gray!20](#1,\fpeval{#1 * 2 + 3},\fpeval{#1 * 2 + 4})
    }
    {nofill}{\tkzDrawPolygon(#1,\fpeval{#1 * 2 + 3},\fpeval{#1 * 2 + 4})}
    }
    \tkzDrawPoints(\fpeval{#1 * 2 + 3},\fpeval{#1 * 2 + 4})
}

\newcommand{\OrigRightNoDots}[4]{
    \tkzDefShiftPoint[#1](0:#2){\fpeval{#1 * 2 + 3}}
    \tkzDefShiftPoint[#1](-60:#2){\fpeval{#1 * 2 + 4}}
    \IfEqCase{#3}{
    {fill}{
    \tkzDrawPolygon[fill = gray!20](#1,\fpeval{#1 * 2 + 3},\fpeval{#1 * 2 + 4})
    }
    {nofill}{\tkzDrawPolygon(#1,\fpeval{#1 * 2 + 3},\fpeval{#1 * 2 + 4})}
    }
    \tkzDrawPoints(\fpeval{#1 * 2 + 3},\fpeval{#1 * 2 + 4})
    \IfEqCase{#4}{
    {end}{
    \tkzLabelPoint[below](\fpeval{#1 * 2 + 4}){$\vdots$}
    }
    {mid}{}
    }
}

\title{On Sum Ranges for $3n$-convergence}

\newif\ifuniqueAffiliation 
\uniqueAffiliationtrue

\ifuniqueAffiliation
\author{ {Preston Martens} \\
	Department of Mathematics\\
	Iowa State University\\
	Ames, IA, 50010 \\
	\texttt{pmart29@iastate.edu} \\
    }
\renewcommand{\headeright}{}
\renewcommand{\undertitle}{}
\renewcommand{\shorttitle}{}

\hypersetup{
pdftitle={On Sum Ranges for 3n-convergence},
pdfsubject={q-math.FA},
pdfauthor={Preston Martens},
pdfkeywords={Sum Range, 3n-convergence, Rearrangement},
}

\begin{document}
\maketitle

\begin{abstract}
	The Riemann Rearrangement Theorem (RRT) tells us that commutativity of infinite series can differ from finite series. We wish to extend the notion of infinite series rearrangements to weaker forms of convergence, such as partial series convergence on every 2nd or 3rd index, called $2n$- or $3n$-convergence. In the case of $2n$-convergence, it has been shown that the sum range, along with the standard cases in the RRT, can produce a shifted additive subgroup of reals. In this paper, we show cases where the $3n$-sum range is still a subgroup and a case where the $3n$-sum range cannot be a subgroup, proving the existence of more depth for $3n$-convergence.
\end{abstract}


\section{Introduction}
The Riemann Rearrangement Theorem (or RRT) states that infinite sums do not necessarily obey commutativity like finite sums. In particular, it can be stated as follows:
\begin{theorem}[Riemann Rearrangement Theorem]
    Let $A = (a_n)_{n \in \N}$ be a sequence of real numbers and $\sum A$ be its sum.
    \begin{enumerate}
        \item[(1)] If $\sum A$ is absolutely convergent, then for any permutation $\pi : \N \to \N$, $\sum_{n=1}^\infty a_{\pi(n)} = x$.
        \item[(2)] If $\sum A$ is conditionally convergent, then for any $y \in \R$, there is a permutation $\pi$ such that $\sum_{n=1}^\infty a_{\pi(n)} = y$.
    \end{enumerate}
\end{theorem}
We call the set of numbers that the series $\sum A$ can be rearranged to its \emph{sum range}, denoted $\SR(A)$. In the first case, $\SR(A) = \{x\}$, and in the second case, $\SR(A) = \R$.

\subsection{2n-convergence}
Sum ranges have been studied for series in other spaces, such as Banach spaces\cite{C-G-K,K-K} and Hilbert spaces\cite{Hadwiger} as well as other notions of convergence, such as $A$-summability \cite{L-Z}, statistical convergence \cite{Fast}, and $C_1$ convergence \cite{B-E}. This work in particular builds on convergence in $\R$ along a filter. Given a sequence $A = (a_n)_{n \in \N}$ with associated sum $\sum A$, we say $\sum A$ converges along the filter $\mathcal{F} = (I_k)_{ \in \N}$ if $\lim_{k \to \infty}\sum_{n=1}^{I_k} a_n$ converges. We start by examining the case where $(I_k)_{k \in \N} = (2k)_{k \in \N}$ and provide an appropriate definition.
\begin{definition}
    We say a series $\sum A = \sum_{n=1}^{\infty} a_n$ \emph{$2n$-converges} to $x \in \R$ if $\lim_{k \to \infty}\sum_{n=1}^{2k}a_n = x$.
\end{definition}
We proceed with a simple example to illustrate the benefits of $2n$-convergence.
\begin{example}
    The so-called Grandi series
    \[\mathcal{G} = 1 + (-1) + 1 + (-1) + \dots\]
    diverges traditionally. But on every even index, the partial sum is always 0. Therefore, along this $2n$-filter, the series converges to 0. Now, we can try to permute this series as in the RRT. In doing so, notice that this series converges only if the tail of the sequence cancels in alternating pairs of 1 and $(-1)$. So we are allowed to prepend an even number of terms, the sum of which must be even. This results in the \emph{2n-sum range}, denoted $\SR_2(\mathcal{G})$, being equal to $2\Z$.
\end{example}

In the example above, the $2n$-sum range has a different form than what is given by the RRT. In particular, $2\Z$ is an additive subgroup of the reals. In fact, Y. Dybskiy and K. Slutsky found that a shifted additive subgroup structure is the only other option for a sum that $2n$-converges aside from the classical results of the RRT\cite{D-S}.
\begin{theorem}[Classification of the $2n$-sum range]
    Let $\sum A = \lim_{k \to \infty} \sum_{n=1}^{2k} a_n$ be a series of real numbers that $2n$-converges to $x$. Then $\SR_2(A)$ is one of the following:
    \begin{enumerate}
        \item[(1)] The whole of $\R$;
        \item[(2)] The singleton $\{x\}$;
        \item[(3)] A shifted additive subgroup of the form $r + \Gamma$, where $r \in \R$ and $\Gamma \subseteq \R$ is any countable additive subgroup. Note that case (2) is a special case where $r = x$ and $\Gamma = \{0\}$.
    \end{enumerate}
\end{theorem}

\subsection{Acknowledgments}
I am grateful for the Iowa State Mathematics Research Teams (ISMaRT) and professor Peter Burton for originally giving me the opportunity to research $3n$-convergence. I am also very appreciative of professor Konstantin Slutsky for being my honors capstone project advisor and providing support, sources, and feedback for my research and paper.

\vskip.1in
\section{3n-convergence}
\label{sec:3n-convergence}
\subsection{Introduction}
In an analogous fashion to $2n$-convergence, we define $3n$-convergence by convergence along the filter $(3k)_{k \in \N}$. We similarly define the $3n$-sum range for a sequence $A$ and denote it $\SR_3(A)$. At the end of their paper, Dybskiy and Slutsky posit that when considering $3n$-convergence, the $3n$-sum range is not a shifted additive subgroup and provide an example. Their example provided therein is in error as the sum range is $\Z$, and the main purpose of this work is to provide a valid example.
\par Before the statement of the main theorem, some terminology.
\begin{definition}[Complements and Zero Triples]
    An element $a$ in a sequence $A$ has a \emph{complementary pair} if there exist distinct $b,c \in A$ such that $a + b + c = 0$. Such a triple is called a \emph{zero triple}. If a sequence $A$ can be partitioned into zero triples, we say $A$ has a \emph{zero triple arrangement} (ZTA).
\end{definition}
\begin{definition}
    A subsequence $S \subseteq A$ is \emph{admissible} if $\sum S = s \in \SR_3(A)$.
\end{definition}

\subsection{Main Theorem}
\begin{theorem}
    Consider a sequence $A \subseteq \Z$ that satisfies:
\begin{enumerate}
    \item[(1)] Each element $a_k \in A$ appears at most once.
    \item[(2)] Each $a_k$ is in exactly two zero triples.
    \item[(3)] Elements have sufficiently fast growth determined by $|a_{3k+1}|,|a_{3k+2}| \geq \sum_{i=1}^{3k} |a_i|$.
\end{enumerate}
Then, $\SR_3(A)$ is not a countable additive subgroup of reals.
\end{theorem}

Remark: We will show this by translating the language of adding triples into a triangular tree graph. This graph is accompanied by a flip operation that sends a pair of adjacent vertices in a subset to a related pair of the same sum. Using this operation, we will exhibit an element in $\SR_3(A)$ that cannot be doubled. Condition (2) guarantees $0 \in \SR_3(A)$, thus proving the Theorem.
\vskip.1in
Although the main theorem provides a specific example where the $3n$-sum range is not a shifted additive subgroup, there are more general cases where we do get a shifted additive subgroup. However, to start to motivate the conditions of the main theorem, we first assume $A \subseteq \Z$ from here on.

\subsection{Cases Yielding a Subgroup}
\begin{definition}
    Let $A$ have the \emph{infinite complementary property} (ICP) if for every $a \in A$, there are infinitely many distinct $b,c \in A$ such that $a + b + c = 0$.
\end{definition}
Note 1: If $A$ has the ICP, then $A$ has a ZTA. Additionally, the ICP is a residual property, meaning for any finite $S \subseteq A$, $A \setminus S$ also has the ICP.
\begin{lemma}
    Let $A$ have the ICP and $a_1 + \dots + a_{3k} \in \SR_3(A)$. Then for every $N \in \N$ there exists $b_1 + \dots + b_{3\ell} = a_1 + \dots + a_{3k}$ such that the index of $b_i \in A$, $\mathrm{ind}(b_i)$, is at least $N$ for all $1 \leq i \leq 3\ell$.
\end{lemma}
\begin{proof}
    Choose for $a_1$ any complementary pair $a_1'$ and $a_2'$. Because $a_1'$ has infinitely many complementary pairs, there must exist one of with both indices at least $N$, say $b_1$ and $b_2$. Using the ICP on $a_2'$, we can choose a complementary pair $b_3$ and $b_4$ such that both indices are strictly greater than $\max\{\textrm{ind}(b_1),\textrm{ind}(b_2)\}$. Thus, $a_1 = b_1 + b_2 + b_3 + b_4$. Proceeding in a similar manner, for each $a_i$ we can find $b_{4i-3}$, $b_{4i-2}$, $b_{4i-1}$, $b_{4i}$ whose sum is $a_i$ and in which every pair of indices is strictly greater than the maximum of all previous indices. In total, we have $4 \cdot (3k) = 3\ell$ terms such that $a_1 + \dots + a_{3k} = b_1 + \dots + b_{3\ell}$ and $\textrm{ind}(b_i) \geq N$ for all $1 \leq i \leq 3\ell$.
\end{proof}
Note 2: This Lemma also applies if $A$ is any general sequence with the ICP. However, the following Theorem requires the extra assumption $A \subseteq \Z$.
\begin{theorem}
    If $A \subseteq \Z$ has the ICP, then $\SR_3(A) = \{a_1 + \dots + a_{3k} : k \in \N, a_i \neq a_j \in A\}$. Furthermore, this forms an additive subgroup in $\Z$.
\end{theorem}
\begin{proof}
    Let $A \subseteq \Z$ have the ICP. For any series in the sum range, trailing triples must sum to zero since $\Z$ is uniformly separated. Therefore, we have $k$ initial triples of freedom for some $k \in \N$, giving the right-hand side. Conversely, given any sum $a_1 + \dots + a_{3k}$, we can take those as the first $k$ triples. Since the ICP is residual, $A \setminus \{a_1,\dots,a_{3k}\}$ has the ICP and therefore a ZTA. Appending this to the first $k$ triples gives a permutation of $A$ with sum $a_1+\dots +a_{3k}$.
    \vskip.1in
    To show this forms a subgroup, we first have $0 \in \SR_3(A)$ by Note 1. Additionally, if $x = a_1 + \dots + a_{3k}, y = b_1 + \dots + b_{3\ell} \in \SR_3(A)$, then by Lemma 2.3.2, we can find $c_1, \dots, c_{3m} \in A$ such that $b_1 + \dots + b_{3\ell} = c_1 + \dots + c_{3m}$ and $\{c_1,\dots,c_{3m}\} \cap \{a_1,\dots,a_{3k}\} = \varnothing$. Therefore, $a_1+\dots + a_{3k} + c_1 + \dots + c_{3m} = x+y \in \SR_3(A)$ and each element is distinct, proving closure under addition. To show $-x \in \SR_3(A)$, we know that there exist $b_1 \neq c_1 \in A$ such that $-a_1 = b_1 + c_1$. By the ICP, $a_2$ has $b_2 \neq c_2 \in A \setminus \{b_1,c_1\}$ such that $-a_2 = b_2 + c_2$. We continue on in a similar manner, producing $6k$ distinct elements with $\sum_{i=1}^{3k} (b_i + c_i) = -\sum_{i=1}^{3k}a_i = -x$. So, $-x \in \SR_3(A)$.
\end{proof}
From this powerful theorem we get various corollaries of note.
\begin{corollary}
    If $A = \Z$, then $\SR_3(A) = \Z$. Similarly, if $A$ is a subgroup of $\Z$ with generator $n$, then $\SR_3(A) = n\Z$.
\end{corollary}
\begin{corollary}
    Let $S \subseteq \Z$ such that $|S| < \infty$. If $A = \Z \setminus S$, then $\SR_3(A) = \Z$.
\end{corollary}
\begin{corollary}
    Let $n \geq 3$. If $S = n\Z$, and $A = \Z \setminus S$, then $\SR_3(A) = \Z$.
\end{corollary}
\subsection{Finite Complements and the Main Theorem}
We now tack on one extra assumption found in the main theorem, namely that each $a_k \in A$ appears exactly once. The preceding section demonstrated the usefulness of the ICP in obtaining a subgroup of integers. To motivate an additional main theorem assumption, first suppose we have a property opposite in a sense to the ICP.
\begin{prop}
    If for every $a \in A \subseteq \Z$, there is exactly one complementary pair, then $\SR_3(A) = 0$.
\end{prop}
\begin{proof}
    Any zero triple in $A$ is the unique complementary pair for all three elements in the triple, so $A$ is partitioned into disjoint zero triples. If $S \subseteq A$ is admissible, then $A \setminus S$ is composed of infinitely many zero triples. Therefore, $S$ is finite and composed of elements from zero triples in some permutation, so $\sum S = 0$.
\end{proof}
This case is useful but not very interesting, so let us go one step further by incrementing the number of complementary pairs.
\begin{definition}
    If every element $a \in A$ has exactly $m$ complementary pairs, we say that $A$ is an $m$-complement set.
\end{definition}
We claim that such a set does exist with an important initial degree of flexibility. However, once that choice is made, the structure of $A$ is restricted.
\begin{theorem}
    We can create infinitely many $m$-complementary sets of integers. Additionally, no two sets of zero triples share more than one element.
\end{theorem}
\begin{proof}
    We show that we can make infinitely many $2$-complement sets of integers, as higher $m$ will follow a similar process. Take some arbitrary zero triple of distinct integers $a_0 + a_1 + a_2 = 0$. To give $a_0$ a second complementary pair, we can choose $a_3 \neq a_4 \in \Z \setminus \{a_0,a_1,a_2\}$ such that $a_0 + a_3 + a_4 = 0$. Let $A_0 = \{a_0,\dots,a_4\}$. By the Pigeonhole Principle, no other triple in $A_0$ can sum to zero. Since $A_0$ is finite, the set $\{|x| : x \in A_0\}$ is bounded above by some $M_0 \in \R$. To give $a_1$ a second complement, we can choose $a_5 \neq a_6 \in \Z \setminus A_0$ such that $a_1 + a_5 + a_6 = 0$ and $|a_5|,|a_6| > 2M_0$. Since $a_5$ and $a_6$ are large, they cannot be repeats. Additionally, neither of them can complement any other existing element in $A_1 = A_0 \cup \{a_5,a_6\}$, thus ensuring we have exactly two complements for $a_0$ and $a_1$. We can continue on in this way, finding for each $a_k$ a complementary pair:
    
    \[a_{2k+3},a_{2k+4} \in \Z \setminus A_{k-1} \textnormal{ such that } |a_{2k+3}|,|a_{2k+4}| > 2M_{k-1},\]
    
    where $A_{k-1} = A_{k-2} \cup \{a_{2k+1},a_{2k+2}\}$ and $M_{k-1} \geq \max\{|x| : x \in A_{k-1}\}$. Since $a_0,a_1,a_2$ were arbitrary, we can construct a branching 2-complement set for any given zero triple.
\end{proof}
Since no two zero triples share more than the given complemented element, we get a branching tree-like structure. We can reformulate this using a triangular graph as shown in Figure 1 above. One triangle corresponds to a zero triple, and two triangles share a vertex if they share the corresponding element in a zero triple. If a zero triple is used in the series expansion of $A$, then we can shade the triangle, and leave it unshaded otherwise. Index labeling starts from the center and proceeds as in the proof of Theorem 2.4.3.
\begin{figure}[H]
\begin{center}
\begin{tikzpicture}[scale=0.5, label style/.style={font=\scriptsize}]
    \tkzDefPoint(0,0){0}
    \tkzDefShiftPoint[0](-60:3){1}
    \tkzDefShiftPoint[0](-120:3){2}
    \tkzDrawPolygon(0,1,2)
    \tkzDrawPoints(0,1,2)

    \OrigTop{0}{3}{nofill}{mid}
    \OrigRight{1}{3}{nofill}{mid}
    \OrigLeft{2}{3}{nofill}{mid}
    
    \FlipLeft{3}{2}{nofill}
    \FlipRight{4}{2}{nofill}
    \FlipRight{5}{2}{nofill}
    \FlipBottom{6}{2}{nofill}
    \FlipLeft{7}{2}{nofill}
    \FlipBottom{8}{2}{nofill}

    \OrigLeft{9}{1}{nofill}{end}
    \OrigTop{10}{1}{nofill}{end}
    \OrigRight{11}{1}{nofill}{end}
    \OrigTop{12}{1}{nofill}{end}
    \OrigRight{13}{1}{nofill}{end}
    \OrigTop{14}{1}{nofill}{end}
    \OrigRightNoDots{15}{1}{nofill}{end}
    \OrigLeft{16}{1}{nofill}{end}
    \OrigLeft{17}{1}{nofill}{end}
    \OrigTop{18}{1}{nofill}{end}
    \OrigRight{19}{1}{nofill}{end}
    \OrigLeft{20}{1}{nofill}{end}

    \tkzLabelPoint[right](0){$a_0$}
    \tkzLabelPoint[above right](1){$a_1$}
    \tkzLabelPoint[above left](2){$a_2$}
    \tkzLabelPoint[below left](3){$a_3$}
    \tkzLabelPoint[below right](4){$a_4$}
    \tkzLabelPoint[below right](5){$a_5$}
    \tkzLabelPoint[right](6){$a_6$}
    \tkzLabelPoint[above right](7){$a_7$}
    \tkzLabelPoint[right](8){$a_8$}
\end{tikzpicture}
\caption{The graph of a 2-complement set $A = \{a_0,a_1,\dots\}$ obtained by the proof in Theorem 2.4.3.}
\end{center}
\end{figure}
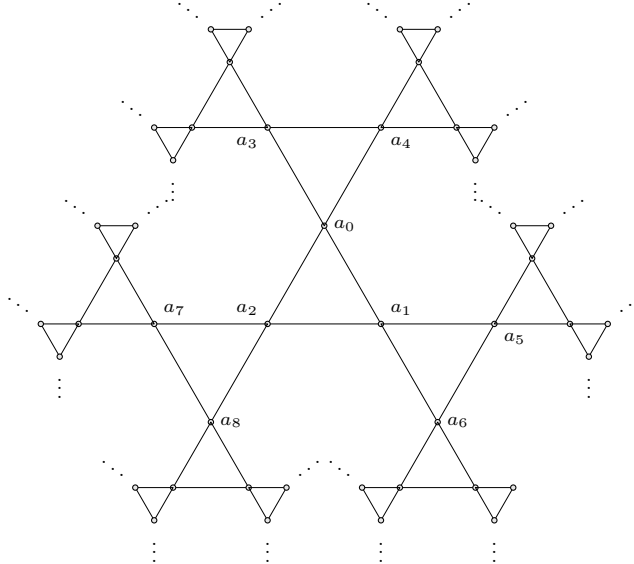
We can now redefine an admissible set $S \subseteq A$ in terms of this triangular graph. A finite set $S$ is admissible if the remaining elements $A \setminus S$ can be shaded by triangles such that each point in $A \setminus S$ is in exactly one shaded triangle. The conclusion then is that admissible sets in $A$ have a correspondence with elements in $\SR_3(A)$ given by the sum of the elements in that admissible set. Reframing admissibility this way, we can define an operation on these sets of points, which we will call the flip operation.
\begin{definition}
    For two adjacent points $a$ and $b$ with complement $c$, we can define a flip in two directions (when applicable):
    \begin{enumerate}
        \item[(1)] If the index of $c$ is greater than those of $a$ and $b$, then send $a$ and $b$ to the other unique pair of points $d$ and $e$ such that $c + d + e = 0$. This is called a flip outwards as the indices of $d$ and $e$ are greater than those of $a$ and $b$.
        \item[(2)] If $a$ and $b$ are not on the starting triangle and the index of $c$ is less than those of $a$ and $b$, then send $a$ and $b$ to the other unique pair of points $d$ and $e$ such that $c + d + e = 0$. This is called a flip inwards as the indices of $d$ and $e$ are less than those of $a$ and $b$.
    \end{enumerate}
\end{definition}
For notational convenience, we may write $F(a_k)$ and $F(a_\ell)$ as the two points $a_k$ and $a_\ell$ flip to, although this is not necessarily unique. It is clear to see that flipping preserves the sum of the set. However, flipping also preserves admissibility of the set.
\begin{prop}
    If $S$ is admissible, then the set $S'$ created by flipping any pair of points inwards or outwards is also admissible.
\end{prop}
\begin{proof}
    Let $S$ be admissible, and suppose we want to flip points $a,b \in S$ outwards to positions $d$ and $e$, with $c$ being the complement of $a,b$ and $d,e$. In this position, $\triangle abc$ is unshaded, as well as the triangles of lower index adjacent to $a$ and $b$, while $\triangle cde$ is shaded as $c \notin S$. After flipping $a$ and $b$ to $d$ and $e$, then $\triangle abc$ must be shaded, while $\triangle cde$ becomes unshaded. Note that the triangles of lower index adjacent to $a$ and $b$ are still unshaded, so the shading of the inner triangles is unaffected by the operation, while the shading of $\triangle cde$ onward is opposite to the original configuration. However, by a parity argument on triangular gaps we will demonstrate in the proof of Theorem 2.5.2, this will not clash with any potentially affected points. So the resultant set after flipping these points is admissible. Admissibility after an inward flip is shown similarly.
\end{proof}
We are now ready to prove the main theorem.
\begin{proof}
    Let $A \subseteq \Z$ satisfy the conditions of the main theorem. That is, each $a_k \in A$ appears at most once, each $a_k$ is in exactly two zero triples, and $|a_{3k+1}|,|a_{3k+2}| \geq \sum_{i=1}^{3k}|a_i|$. Let $S = \{a_0,a_1,a_3\}$ and $s = a_0 + a_1 + a_3$, which is admissible. Assume there is a set $S'$ with sum $2s$ and least maximum index. Due to the rapid growth in assumption 3, this maximum index must be at least 5. For such an arrangement to be possible, there must now be a pair of maximal elements $a_{n-1}$ and $a_n$, which we can flip inwards to a pair with equivalent sum as shown in Figure 2. Since flipping a pair preserves admissibility, this resultant set is still admissible. However, the result of the flip itself must not lie in $S'$ because each element appears exactly once. We have now reduced the set $S'$ to a set with the same sum but smaller maximum index, contradicting the choice of $S'$.
    \end{proof}
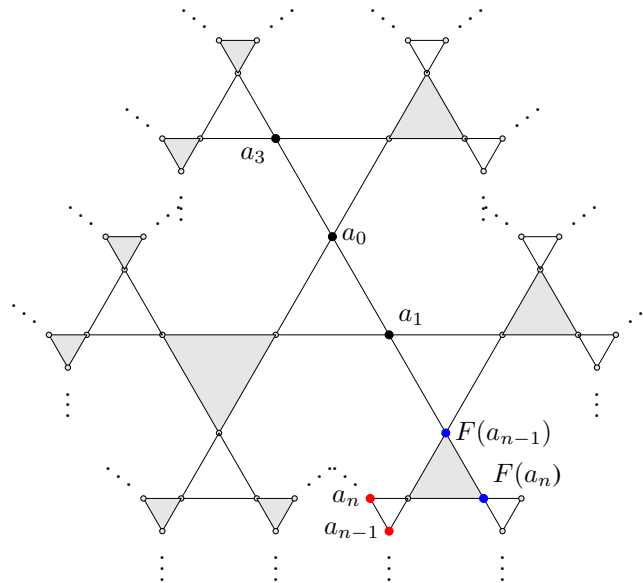
\begin{figure}[H]
\begin{center}
\begin{tikzpicture}[scale=0.5, label style/.style={font=\normalsize}]
    \tkzDefPoint(0,0){0}
    \tkzDefShiftPoint[0](-60:3){1}
    \tkzDefShiftPoint[0](-120:3){2}
    \tkzDrawPolygon(0,1,2)
    \tkzDrawPoints(0,1,2)

    \OrigTop{0}{3}{nofill}{mid}
    \OrigRight{1}{3}{nofill}{mid}
    \OrigLeft{2}{3}{fill}{mid}
    
    \FlipLeft{3}{2}{nofill}
    \FlipRight{4}{2}{fill}
    \FlipRight{5}{2}{fill}
    \FlipBottom{6}{2}{fill}
    \FlipLeft{7}{2}{nofill}
    \FlipBottom{8}{2}{nofill}

    \OrigLeft{9}{1}{fill}{end}
    \OrigTop{10}{1}{fill}{end}
    \OrigRight{11}{1}{nofill}{end}
    \OrigTop{12}{1}{nofill}{end}
    \OrigRight{13}{1}{nofill}{end}
    \OrigTop{14}{1}{nofill}{end}
    \OrigRightNoDots{15}{1}{nofill}{end}
    \OrigLeft{16}{1}{nofill}{end}
    \OrigLeft{17}{1}{fill}{end}
    \OrigTop{18}{1}{fill}{end}
    \OrigRight{19}{1}{fill}{end}
    \OrigLeft{20}{1}{fill}{end}

    \begin{scope}
        \tkzSetUpPoint[size = 3, color = black, fill = black]
        \tkzDrawPoints(0,1,3)
    \end{scope}
    \tkzLabelPoint[right](0){$a_0$}
    \tkzLabelPoint[above right](1){$a_1$}
    \tkzLabelPoint[below left](3){$a_3$}
    \begin{scope}
        \tkzSetUpPoint[size = 3, color = red, fill = red]
        \tkzDrawPoints(35,36)
    \end{scope}
    \tkzLabelPoint[left](35){$a_n$}
    \tkzLabelPoint[left](36){$a_{n-1}$}
    \begin{scope}
        \tkzSetUpPoint[size = 3, color = blue, fill = blue]
        \tkzDrawPoints(6,15)
    \end{scope}
    \tkzLabelPoint[right](6){$F(a_{n-1})$}
    \tkzLabelPoint[above right](15){$F(a_n$)}
\end{tikzpicture}
\caption{Demonstrating flipping the maximal elements of $S'$ inwards given the initial set $S$}
\end{center}
\end{figure}
\newpage
\subsection{Additional Structural Results in 2-complement Sets}
We have just shown an example of a $3n$-sum range that is not a subgroup by examining a sum range in 2-complement set that fails closure under addition. That being said, we can still find additive inverses, which we first show by geometrically characterizing the admissible subsets.
\vskip.1in
\begin{definition}[Chains and Triangular Gaps]
    A \emph{chain} is a set of vertices that form a path of any nonnegative length. Given two vertices in $a_i,a_j \in S$ with $i < j$, a \emph{triangular gap}, or gap for short, occurs between them if the triangle adjacent to $a_i$ branching to $a_j$ does not contain any points in $S$, while the succeeding triangle branching to $a_j$ contains $a_j$. Gaps of $n$ triangles occur between vertices $a_i$ and $a_j$, $i < j$, if the triangle adjacent to $a_i$ branching to $a_j$ does not contain any points in $S$, and the (unique) vertex that is both in this triangle and the one containing $a_i$ has a gap of $n-1$ triangles to $a_j$.
\end{definition}
\begin{theorem}[Characterization of Admissible Subsets of a 2-complement Set]
    A subset $S \subseteq A$ is admissible if and only if it is composed of chains of points separated by gaps of an odd number of triangles.
\end{theorem}
\begin{proof}
We can partition any admissible set $S$ into chains because individual vertices are chains by definition. If $S$ had two chains not separated by a triangular gap, then there are two adjacent triangles, both containing either one or two vertices in $S$. The vertex in both triangles must not be in $S$, otherwise we have only one chain. But then at least one of these triangles must be shaded, using an unavailable element in $S$. Since shaded and unshaded triangles must alternate, the same holds for any gap of even parity. Conversely, any set of chains with gaps of one triangle is admissible by shading the triangular gaps, producing a ZTA of $A \setminus S$. Since shading alternates parity, the same holds for any gaps of odd parity.
\end{proof}
Indeed, the set $S = \{a_0,a_1,a_3\}$ in the proof of the main theorem is itself a chain of three vertices. But we can have other admissible sets such as $\{a_3,a_4,a_6,a_7\}$ using the vertex labeling in Figure 1. We could not, however, choose only $\{a_1,a_3\}$, as there is a gap of zero triangles between the two. This language of triangular gaps is used to justify the result that $a_0$ in this example cannot be reached by $S$ or a ZTA of $A \setminus S$. Using this helpful fact, we will now show that even though the $3n$-sum range of series in the main theorem may not be closed under addition, they are always closed under negation.
\begin{prop}
    If $A \subseteq \Z$ is a 2-complement set such that each $a_k$ appears exactly once, then $\SR_3(A)$ is symmetric. That is, if $s \in \SR_3(A)$, then $-s \in \SR_3(A)$ as well.
\end{prop}
\begin{proof}
    Let $s \in \SR_3(A)$ with $S$ being the admissible set such that its sum is $s$. By Theorem 2.5.2, $S$ is composed of chains of vertices with triangular gaps of odd parity. Let $S'$ be the set of vertices achieved by negating every point in $S$ outwards; that is, taking each point in $S$ to its complementary pair of higher index. By construction, $\sum S' = -s$. Suppose $S'$ was not admissible. By reducing shading parity, we can assume $S'$ has two chains with no triangular gap. Then there are two adjacent triangles with vertices $a_i, a_{2i+3}, a_{2i+4}, a_{2(2i+4) + 3}, a_{2(2i+4)+4}$ given by  the construction in Theorem 2.4.3, which we will relabel as $s_1,\dots,s_5$ respectively. Note that $s_3$ is in both triangles, so it cannot be in $S'$. The higher-indexed chain $s_4$ and $s_5$ can only be achieved by negating $s_3$, so $s_3 \in S$, meaning both $s_4,s_5 \in S'$. If $s_2 \in S'$, then $s_1 \in S$. But negating $s_1$ and $s_3$ produces $\{s_2,s_3,s_4,s_5\}$, and $\{s_3,s_4,s_5\}$ forms a triangle, so we reduce it, a contradiction. If $s_1 = a_i \in S'$, then it has a vertex, say $a_{\frac{i-3}{2}}$, from a previous triangle that negates to it in $S$. But now this triangle and the triangle $\{s_1,s_2,s_3\}$ are adjacent and have chains with no gap, contradicting the admissibility of $S$. Therefore, $S'$ must be admissible.
\end{proof}

\subsection{Further Questions}
It is apparent that these assumptions, although useful, are very specific. The first natural place to look to generalize these arguments is $m$-complement sets for a general $m$. However, one important distinction is the fact that in a 2-complement set, a pair is mapped to a unique pair by flipping it. We lose this uniqueness for higher $m$.
\begin{question}
    What are the properties of admissible subsets for an $m$-complement set? What if each $a_k \in A$ is allowed to have a varying number of complements?
\end{question}
We can also look to examples of $kn$-convergence for $k > 3$. In these cases, the graph unit becomes a regular $k$-gon instead of a triangle. Note that in $3n$-convergence, any pair has a unique element that can complement it, which we lose for higher $k$.
\begin{question}
    Are the properties of 2-complement (or more generally $m$-complement) sets still the same for $kn$-convergence?
\end{question}
We are also working with the assumption that each $a_k$ appears exactly once, $A$ is uniformly separated, and $A$ is unbounded. Changing any or all of these factors is likely to produce different structures of the $3n$-sum range. In particular, it would be interesting to find an example where the $3n$-sum range is bounded, say from below. In any case, there is still much to explore here as we look to the goal of classifying all possible $3n$-sum ranges.

\bibliographystyle{plain}
\bibliography{References}

\end{document}

\typeout{get arXiv to do 4 passes: Label(s) may have changed. Rerun}